\documentclass[12pt]{article}
\usepackage{amssymb,a4}
\usepackage{amsthm,amsmath,anysize}
\def\a{\alpha}
\def\l{\lambda}
\def\g{\gamma}
\def\0{\bar{0}}
\def\e{\epsilon}
\def\d{\delta}
\def\g{\mathfrak{g}}
\newtheorem{lemma}{Lemma}[section]
\newtheorem{theorem}[lemma]{Theorem}

\newtheorem{proposition}[lemma]{Proposition}
\newtheorem{definition}[lemma]{Definition}
\newtheorem{corollary}[lemma]{Corollary}

\hoffset \voffset \oddsidemargin=60pt \evensidemargin=45pt
\title{\bf  On simple modules for the  restricted Lie superalgebra $gl(m|n)$}
\author{Chaowen Zhang \\College of Sciences,\\
 China University of Mining and Technology, \\Xuzhou 221008, China}
\date{ }

\begin{document}
\maketitle

\begin{abstract}  In this paper, we study the simple modules for the restricted Lie superalgebra
$gl(m|n)$. A condition for the simplicity of the induced modules is
given, and an analogue of Kac-Weisfeiler theorem is proved.
\end{abstract}


\smallskip\bigskip

\section{Introduction}

The notion Lie superalgebras was defined by Kac in \cite{k1}.  It is
a generalization of Lie algebras.  Let $\mathbf F$ be an
algebraically closed field with  char $\mathbf{F}=p>0$. Lie
superalgebras over $\mathbf F$ were studied in \cite{kl} and
\cite{p}. Also the notion of restricted Lie superalgebras was given.
\par  Let $\g=\g_{\0}\oplus \g_{\bar 1}$ be a restricted Lie
superalgebra over $\mathbf F$. Then  the p-characters were
introduced for the $\g$-modules in \cite{zl}.   It is simply a
linear function on the even part $\g_{\0}$. Each simple $\g$-module
has a p-character $\chi$(\cite{zl}).
\par
 The present paper is addressing the following
questions: \par 1. What is the possible maximal dimension of each
simple $u(\mathfrak{g},\chi)$-module?\par 2. How is a simple
$u(\mathfrak{g},\chi)$-module  related to its simple
$u(\mathfrak{g}_{\bar{0}},\chi)$-submodules? \par 3. In studying a
simple $u(\mathfrak{g},\chi)$-module, can its p-character $\chi$, as
in the case of classical Lie algebras, be reduced to a nilpotent
one?
\par
 The
paper is organized as follows: Section 2 gives the preliminaries. In
Section 3, we work on Question 1 and study the condition for the
simple $u(\g,\chi)$-module to attain the maximal possible dimension.
Section 4 deals with Question 2. We study the simple modules for
$gl(m|n)$ and determine the condition for the {\it graded baby verma
module} to be simple.\par In  Section 5  we work on Questions 3. We
have proved a version of the Kac-Weisfeiler Theorem for
$\g=gl(m|n)$: Each simple $u(\g,\chi)$-module is induced by a simple
$u(\mathfrak{l},\chi)-submodule$, where $\chi_{|\mathfrak{l}}$ is
nilpotent. Then we discuss the properties of $Z^{\chi}(\l)$ when
$\chi$ is in the standard Levi form.\par After the first submission
of the present paper, I received from Weiqiang Wang
 their independent work on modular representations of
Lie superalgebras: {\sl Representations of Lie superalgebras in
prime characteristic I}. I would like to thank him for the comments.
\section{Preliminaries}
\subsection{Basic definitions} \noindent Let $\mathbf{F}$ be an
algebraically closed field with char $\mathbf{F}=p>0$. Assume
$\mathbf{Z}_{2}=\mathbf{Z}/2\mathbf{Z}=\{\overline{0},
\overline{1}\}$.
  Let $V=V_{\overline{0}}\oplus V_{\overline{1}}$ be a $\mathbf{Z}_{2}$-graded
vector space over $\mathbf{F}$. We denote by $\mathrm{p}(a)=\theta$
the parity of a homogeneous element $a\in V_{\theta}, \theta\in
\mathbf{Z}_{2}.$ We assume throughout that the symbol
$\mathrm{p}(x)$ implies that $x$ is $\mathbf{Z}_2$-homogeneous.

 A \textit{superalgebra} \label{superalgebra} is a $\mathbf{Z}_{2}$-graded vector
space $\mathcal{A}=\mathcal{A}_{\overline{0}}\oplus
\mathcal{A}_{\overline{1}}$ endowed with an algebra structure such
that $\mathcal{A}_{\theta}\mathcal{A}_{\mu}\subset
\mathcal{A}_{\theta+\mu} $ for all $\theta,\mu\in \mathbf{Z}_{2}.$
 A superalgebra  $\mathfrak{g}=\mathfrak{g}_{\overline{0}}\oplus \mathfrak{g}_{\overline{1}}$
 over $\mathbf{F}$  is called a {Lie superalgebra}
provided that
\begin{enumerate}
 \item[$\mathrm{(i)}$]

  $[a,b]=-(-1)^{\mathrm{p}(a)\mathrm{p}(b)}[ b,a],$

\item[$\mathrm{(ii)}$] $[a,[b,c]]=[[a,b],c]+(-1)^{\mathrm{p}(a)\mathrm{p}(b)}[b,[a,c]],$
\end{enumerate}
 for  $a,b\in \mathfrak{g}_{\overline{0}}\cup \mathfrak{g}_{\overline{1}}, c\in \mathfrak{g}.$

Let $\mathfrak{g}=\mathfrak{g}_{\overline{0}}\oplus
\mathfrak{g}_{\overline{1}} $ be a Lie superalgebra. Then the even
part $\mathfrak{g}_{\overline{0}}$ is a Lie algebra and the odd part
$\mathfrak{g}_{\overline{1}}$ is a
$\mathfrak{g}_{\overline{0}}$-module
 under the adjoint action. Note that in the case {char} $\mathbf{F}=2,$ a Lie
superalgebra is a $\mathbf{Z}_{2}$-graded Lie algebra. Thus one
usually adopts the convention that $\mathrm{char}\ \mathbf{F}=p>2 $
in the  modular case.

Let $V=V_{\overline{0}}\oplus V_{\overline{1}} $ be a
$\mathbf{Z}_{2}$-graded vector space, $\dim V_{\overline{0}}=m,$
$\dim V_{\overline{1}}=n.$ The algebra
$\mathrm{End}_{\mathbf{F}}(V)$
  becomes an associative superalgebra if one
defines
$$\mathrm{End}_{\mathbf{F}} (V)_{\theta }
:=\{A\in \mathrm{End}_{\mathbf{F}} (V)\mid A (V_{\mu})\subset
V_{\theta +\mu}, \mu\in \mathbf{Z}_{2}\}
$$
for  $\theta\in \mathbf{Z}_{2}.$ On the vector superspace
$\mathrm{End}_{\mathbf{F}} (V) =\mathrm{End}_{\mathbf{F}}
(V)_{\overline{0}}\oplus \mathrm{End}_{\mathbf{F}}
(V)_{\overline{1}}$, we define a new multiplication $[\ ,\ ]$ by
$$ [A,B]:=AB-(-1)^{\mathrm{p}(A)\mathrm{p}(B)}BA
\quad \mbox{for  }\ A, B\in \mathrm{End}_{\mathbf{F}} (V).$$ This
superalgebra endowed with the new multiplication is a Lie
superalgebra, denoted by
 $$ {gl}(V)={gl}_{\overline{0}}(V)\oplus {gl}_{\overline{1}}(V)
 \mbox{or}\quad
 gl(m|n)=gl_{\overline{0}}(m|n)\oplus gl_{\overline{1}}(m|n).$$ For
  \[A=\begin{pmatrix}
 \alpha & \beta\\
 \gamma & \delta
       \end{pmatrix}\in gl(m|n),\] define the supertrace:
$$\mathrm{str}(A)=\mathrm{tr}(\alpha)-\mathrm{tr}(\delta).$$ Then
$$sl(m|n):=\{A\in gl(m|n)\mid\mathrm{str}(A)=0\}$$ is an ideal in
$gl(m|n)$ of codimension 1.  One can find further information on Lie
superalgebras in \cite{k1,k2,sc}.

A Lie superalgebra $\mathfrak{g}=\mathfrak{g}_{\overline{0}}\oplus
\mathfrak{g}_{\overline{1}} $ is called restricted if
$\mathfrak{g}_{\overline{0}}$ is a restricted Lie algebra and if
$\mathfrak{g}_{\overline{1}}$ is a restricted
$\mathfrak{g}_{\overline{0}}$-module (see \cite{kl,p}). Let
$\mathfrak{g}=\mathfrak{g}_{\overline{0}}\oplus
\mathfrak{g}_{\overline{1}}$ be a restricted Lie superalgebra. The
$p$-mapping $[p]:\mathfrak{g}_{\overline{0}}\rightarrow
\mathfrak{g}_{\overline{0}} $ is also called the $p$-mapping of  the
Lie superalgebra $\mathfrak{g}$. Note that $gl(m|n)$ is restricted
with the usual $p$-mapping ($p$-th power taken in
$\mathrm{End}(V)$).

\subsection{General properties }

Let $\g=\g_{\bar{0}}\oplus \g_{\bar{1}}$ be a Lie superalgebra and
$V=V_{\bar{0}}\oplus V_{\bar 1}$  a $\mathbf{Z}_2$-graded vector
space. If there is an even homomorphism $\rho$ from $\g$ to
$gl(V)$(by "even" we mean that $\rho(\g_{\bar i})\subseteq
gl(V)_{\bar i}$, $\bar i\in \mathbf Z_2$), then we call $V$ a
$\g$-module. A $\g$-module is called simple if it does not have any
proper $\mathbf Z_2$-graded submodules.

\begin{lemma}\cite{zl}
(1) Let $\mathfrak{g}$ be a restricted Lie superalgebra and $M$ a
simple $\mathfrak{g}$-module. Then there is a unique
$\chi\in\mathfrak{g}_{\overline{0}}^{\ast}$ such that
$(x^{p}-x^{[p]}-\chi(x)^p \cdot  1)    M=0$ for all $x\in
\mathfrak{g}_{\overline{0}}.$\par
 (2) Let $I$ be a finite dimensional $\mathbf Z_2$-graded
ideal of a Lie superalgebra $\g$ and $M$ a simple $\g$-module. If
$\rho(x)$ is nilpotent for all $x\in I$, then $IM=0$.
\end{lemma}
Let $\mathfrak{g}$ be a restricted Lie sueralgebra and $U(\g)$ its
universal enveloping algebra. For each $\chi\in
\mathfrak{g}_{\0}^*$, define the $\chi$-reduced enveloping algebra
of $\mathfrak{g}$ by
$u(\mathfrak{g},\chi)=U({\mathfrak{g}})/I_{\chi}$, where $I_{\chi}$
is the $\mathbf{Z}_2$-graded two sided ideal of $U({\mathfrak{g}})$
generated by elements $\{x^p-x^{[p]}-\chi(x)^p1|x\in
\mathfrak{g}_{\overline{0}}\}$. When $\chi=0,$ $u(\mathfrak{g},0)$
is called the restricted universal enveloping algebra of
$\mathfrak{g}$ and simply denoted  by $u(\g)$.

Let $(\mathfrak{g},[p]) $ be a restricted Lie superalgebra. Suppose
that $ u_{1},\ldots, u_{m} $ and $v_{1},\ldots,v_{n}$ are ordered
bases of $\mathfrak{g}_{\overline{0}}$ and
$\mathfrak{g}_{\overline{1}}$ respectively.  Applying a similar
arguments as that for the modular Lie algebra case, we have that
$u(\mathfrak{g},\chi)$ has the following PBW-basis:
$$ \{ v_{1}^{b(1)}\cdots v_n^{b(n)}u_1^{a(1)}\cdots u_m^{a(m)} \mid   0\leq a(i)\leq p-1;
 b(j)=0 \ \mbox{or}\ 1 \}.$$

 A $\mathfrak{g}$-module $M$ is called  having $p$-character
$\chi\in \mathfrak{g}^*_{\overline{0}}$
 provided that
$$x^{p}\cdot m-x^{[p]}\cdot m=\chi (x) ^{p}\ m\quad \mbox{for  all}\,\, x\in
\mathfrak{g}_{\overline{0}},\, m\in M.
$$ Clearly, $u(\mathfrak{g},\chi)$-modules
may be identified with  $\mathfrak{g}$-modules having character
$\chi.$

Let $\Phi=\Phi^+\cup \Phi^-$ be the root system of $\g$. We also
write $\Phi=\Phi_0\cup \Phi_1$, where $\Phi_0$ and $\Phi_1$ are the
set of even and odd roots respectively(\cite{k1}). We denote
$$
\mathfrak{N}^+=\oplus_{\alpha\in \Phi^+} \g_{\alpha},
\mathfrak{N}^-=\oplus_{\alpha\in \Phi^+} \g_{-\alpha}\quad
\text{and}\quad  \mathfrak{B}={H}+\mathfrak{N}^+.$$ As a subalgebra,
$\g$ inherits a natural $p$-mapping from $gl(m|n)$. We denote
$$\rm{Aut}(\mathfrak{g})=\{\theta|\theta \text{ is an even automorphism
of}\quad \g, \theta(x^{[p]})=\theta(x)^{[p]}\quad \text{for
any}\quad x\in \g_{\0}\}.$$

Let $\g$ be one of the following classical Lie superalgebras
\cite[Prop. 2.1.2]{k1}:
$$
\begin{pmatrix}
\g& gl(m|n) & B(m|n) & C_n \\
\g_{\0}& \mathrm{gl}(m)\oplus \mathrm{gl}(n)&\mathrm{B}_m\oplus
\mathrm{C}_n&\mathrm{C}_{n-1}\oplus
\mathbf{F}\\
(m,n)& (m,n)\neq (1,1)&n\geq 2&n\geq 3
\end{pmatrix}
$$
$$
\begin{pmatrix} \g& D(m|n)& F(4) &G(3)\\
   \g_{\0}& \mathrm{D}_m\oplus \mathrm{C}_n&\mathrm{B}_3\oplus \mathrm{A}_1&\mathrm{G}_2\oplus \mathrm{A}_1\\
   (m,n)&m\geq 4,n\geq 2&--&--
\end{pmatrix}
$$
 By
\cite[Prop. 2.5.5]{k1}, we have $\text{dim}\g_{\a}=1$ for any $\a\in
\Phi$. Let
$$ \g'_{\0}=\begin{cases}
C_{n-1}, &\text{if $\g=C_n$}\\
\g_{\0},&\text{otherwise.} \end{cases}$$ Let $G$ be the subgroup of
$\text{Aut}(\g'_{\0})$ such that $Lie(G)=\g'_{\0}$. In particular,
we take $G=GL(m)\times GL(n)$ for $gl(m|n)$. If $\g\neq gl(m|n)$,
let $G$  be the Chevalley group of $\g'_{\0}$.  Each element of $G$
can be  extended naturally to an even automorphism of $\g$. So we
identify $G$ with the subgroup of $\text{Aut}(\g)$ consisting of
elements extended from those of $G$.  By \cite[p.14]{j}, our
restriction on $(m,n)$ ensures that there is a non-degenerate
$G$-invariant bilinear form on $\g'_{\0}$ if we assume $p>3$.\par
 Let
$\chi\in \mathfrak{g_{\0}}^{\ast}$. For the convenience, we
sometimes regard  $\chi$ as a linear function on $\g$ by letting
$\chi(\g_{\bar{1}})=0$.
\begin{lemma}\cite[p.14]{j}
With the assumption on $p$ as above,  let $\g$ be one of the Lie
superalgebras listed. Each $\chi\in \g_{\0}^*$ is conjugate under
$G$ to an element $\chi'\in \g^*_{\0}$ with
$\chi'(\mathfrak{N}^+)=0$.
\end{lemma}

  Let  $\Psi\in G$ such that $\chi^{\Psi}(\mathfrak{N}^{+})=0$.
   If $M=M_{\bar{0}}\oplus M_{\bar{1}}$ is a $\mathfrak{g}$-module, we
denote by $M^{\Psi}$ the $\mathfrak{g}$-module having  $M$ as its
underlying vector space and a new $\mathfrak{g}$-action given by
$xm=\Psi (x)m$ for $x\in \mathfrak{g}$ and $m\in M$, where the
action on the right is the given one. Then
$M^{\Psi}_{\bar{i}}=M_{\bar{i}}$, $i=1,2$, and  $M$ is simple if and
only if $M^{\Psi}$ is. So we may assume $\chi (\mathfrak{N}^+)=0$ in
the following. \par  For each $\chi\in
\mathfrak{g}_{\overline{0}}^{*}$, let $\lambda \in {H}^*$ satisfy
$\lambda(h)^p-\lambda(h)=\chi^p(h)$ for every $h\in {H}$. Define
$\mathbf{F} v$ to be the one dimensional
$u(\mathfrak{B},\chi)$-module as follows:\par
$$
\mathfrak{N}^+ v=0, \quad  h v=\lambda(h)v\quad \text{for
every}\quad   h\in {H}.
$$
 Denote $Z^{\chi}(\lambda)=:U(\mathfrak{g},\chi)\otimes_{U(\mathfrak{B},\chi)}
 \mathbf{F}v.$ $Z^{\chi}(\lambda)$ is called a baby Verma module
 with character $\chi$.  Let $\mathfrak{N}^-=\oplus_{\a<0}\g_{\a}
 =\mathfrak{N}^-_{\bar{0}}+
 \mathfrak{N}^-_{\bar{1}}$, where $\mathfrak{N}^-_{\bar{i}}=\oplus
 _{\a<0,a\in \Phi_i}\g_{\a}$, $i=0,1$.

 Assume $\mathfrak{N}^-_{\bar{1}}$ has a basis $
 f_1,\dots,f_{k}$. Then the PBW theorem shows that
 $$Z^{\chi}(\lambda)=\sum_{1\leq i_1<\dots<i_s\leq k}f_{i_1}\cdots
 f_{i_s}u(\mathfrak{N}^-_{\bar{0}},\chi)v.$$

 We denote $$Z_{\bar{0}}=:\sum_{s \quad\text{is even}}f_{i_1}\cdots
 f_{i_s}u(\mathfrak{N}^-_{\bar{0}},\chi)v,\quad  Z_{\bar{1}}=:\sum_{s \quad\text{is odd}}f_{i_1}\cdots
 f_{i_s}u(\mathfrak{N}^-_{\bar{0}},\chi)v.$$
 Then clearly $Z^{\chi}(\lambda)=Z_{\bar{0}} \oplus Z_{\bar{1}}$ is
a  $\mathbf Z_2$-graded $u(\g,\chi)$-module.

\begin{definition}Let $M=M_{\bar{0}}\oplus M_{\bar{1}}$ be a $u(\mathfrak{g},\chi)$-module.
Let  ${H}$ be a maximal torus contained in the Borel subalgebra
$\mathfrak{B}={H}\oplus \mathfrak{N}^+$. For $\lambda\in H^*$, if
there is a nonzero $v\in M_{\bar{0}}\cup M_{\bar{1}}$ such that
$$\mathfrak{N}^+\cdot v=0, h\cdot v=\lambda (h) v\quad \text{for any}\quad
h\in H,$$ then we say that $v$ is  a maximal vector of highest
weight $\lambda$.
\end{definition}

By \cite{zl}, any simple $u(\g,\chi)$-module contains at least one
maximal vector.

\begin{lemma}\cite{zl}The following are equivalent:\par
(1) $M$ is a simple $u(\g,\chi)$-module.\par (2) Any maximal vector
$v\in M$ generates $M$.
\end{lemma}
\begin{lemma}(1)\cite{zl} If $M$ is a simple $u(\g,\chi)$-module, then $M$ is the
quotient of some $Z^{\chi}(\lambda)$.\par (2) If $\chi$ is
semisimple, then $Z^{\chi}(\lambda)$ has a unique $\mathbf{Z}_2$-
graded maximal submodule.
\end{lemma}
\begin{proof}
(2) Let $V\subseteq Z^{\chi}(\lambda)$ be a proper submodule. Since
the maximal vector $v$ generates $Z^{\chi}(\l)$, we have $v\notin
V$. Suppose there is $f$ in
$u(\mathfrak{N}^-,\chi)\mathfrak{N}^-\subseteq u(\g,\chi)$ such that
$v+fv\in V$. Since $\chi$ is semisimple, Lemma 2.1 implies that $f$
acts nilpotently.  Thus $1+f$ is invertible in $gl(Z^{\chi}(\l))$.
This  leads to $v\in V$, a contradiction. So we have $V\subseteq
u(\mathfrak{N}^-,\chi)\mathfrak{N}^-v$. It follows that the sum of
all the
 proper $\mathbf Z_2$-graded submodules is the unique maximal $\mathbf Z_2$-graded submodule of $M$.

\end{proof}

 Applying a similar arguments as that used
in \cite[Prop.3]{r}, we get
\begin{proposition}If $\chi$ is semisimple, the any simple
$u(\g,\chi)$-module has a unique  maximal  vector.

 \end{proposition}

\section{The simplicity of $Z^{\chi}(\l)$}

Let $\g$ be one of the Lie superalgebras listed in last section. Let
$\Phi=\Phi_{0}\cup\Phi_{1}$ be the root system of $\g$. In this
section, we assume $\chi$ is semisimple. i.e., $\chi(\g_{\a})=0$ for
any $\a\in \Phi$. \par  Let $e_{\a_i}$(resp., $f_{\alpha_i}),
\alpha_i\in \Phi^+$ denote the positive (resp., negative) root
vectors. Define a function $\bar p$: $\Phi^+\longrightarrow \mathbf
Z$  such that
$$\bar{p}(\a_i)=\begin{cases}
p, &\text{if $\alpha_i\in \Phi_{\0}$}\\
2, &\text{if $\alpha_i\in \Phi_{\bar{1}}$.}\end{cases}$$ We write
$\bar p(\a_i)$ simply as $\bar p$.  We abuse the notation $\mathbf
F_2$ as the subset $\{\bar 0, \bar 1\}\subseteq \mathbf F_p$.\par
Assume $|\Phi^+|=n$. Put elements in $\Phi^+$ in the order of
ascending heights: $\a_1,\a_2,\dots, \a_n$. Let $(i_1,\dots i_n)$
and $(i_1,\dots,i_n)^-$, $0\leq i_k\leq \bar p-1$ denote the product
$$e^{i_1}_{\a_1}e^{i_2}_{\a_2}\dots e^{i_n}_{\a_n} \quad
\text{and}\quad f^{i_1}_{\a_1}f^{i_2}_{\a_2}\dots f^{i_n}_{\a_n}\in
u(\g,\chi)$$ respectively. Then using a similar argument as that for
\cite[Prop.4]{r}, we have
\begin{proposition}Any nonzero submodule of $Z^{\chi}(\lambda)$
contains $(\bar{p}-1,\dots,\bar{p}-1)^-v$.
\end{proposition}
Let $\g$ be a restricted Lie superalgebra. Consider $U(\g)$ as a
$\g_{\0}$ -module with the adjoint action, then we get
$$U(\g)=\oplus _{\lambda \in H^*}U(\g)_{\lambda}.$$ We have
$U(\g)_{\lambda}U(\g)_{\mu}\subseteq U(\g)_{\lambda+\mu}$. Then
$U(\g)_0$ is a subalgebra of $U(\g)$.
\begin{lemma}
\cite[7.4.2]{d} Let $L=U(\g)\mathfrak{N}^+\cap U(\g)_0$. Then \par
(1) $L=\mathfrak{N}^-U(\g)\cap U(\g)_0$ and $L$ is a two-sided ideal
of $U(\g)_0$.\par (2) $U(\g)_0=U(H)\oplus L$.
\end{lemma}
Let $\pi$: $U(\g)\longrightarrow u(\g,\chi)$ be the canonical
epimorphism of associative algebras. Then $$\pi(L)\subseteq
\pi(U(\g)\mathfrak{N}^+)=u(\g,\chi)\mathfrak{N}^+.$$ Take a basis of
$\g$:
$$f_{\a_1},\dots,f_{\a_n},h_1,\dots,h_s,e_{\a_1},\dots,e_{\a_n},$$
where the positive roots $\a_1,\dots,\a_n$ is in the order of
descending heights. Then $u(\g,\chi)$ has a PBW basis $$
f^{i_1}_{\a_1}\cdots f^{i_n}_{\a_n}h^{k_1}_1\cdots h^{k_s}_s
e^{j_1}_{\a_1}\cdots e^{j_n}_{\a_n}, 0\leq i_l, j_l\leq \bar p-1,
0\leq k_i\leq p-1.$$ Since $\chi(\mathfrak{N}^+)=0$, so that
$e^{\bar p}_{\a_l}=0$ for $l=1,\dots,n$, each nonzero element of
$\pi(U(\g)\mathfrak{N}^+)$ is a linear combination of the basis
vectors of $u(\g,\chi)$ with $\sum j_l>0$. Thus we get
$$\pi(L)\cap \pi(U(H))=\pi(L)\cap u(H,\chi)=0.$$ It follows from Lemma
3.2(2) that $\pi(U(\g)_0)=u(H,\chi)\oplus \pi(L)$. We call the
projection map from $\pi(U(\g)_0)$ to $u(H,\chi)$ with the kernel
$\pi(L)$ the Harish-Chandra homomorphism, denoted by $\gamma$.\par
For the maximal vector $v$ which defines $Z^{\chi}(\lambda)$,  we
have
$$
(\bar{p}-1,\dots, \bar{p}-1)(\bar{p}-1,\dots, \bar{p}-1)^-v
$$
$$
=\gamma (\Pi e^{\bar{p}-1}_{\alpha_i}\Pi f^{\bar{p}-1}_{\alpha_i})v
$$
$$
=f({\lambda})v.
$$
Note that the definition of $f(\lambda)$ is dependant on the
positive root system $\Phi^+$ and the corresponding Borel
subalgebra.
\begin{theorem}  $Z^{\chi}(\lambda)$ is simple
if and only if  $f(\lambda)\neq 0$.\end{theorem}

\begin{proof} Suppose $f(\lambda)\neq 0$. By Proposition 3.1, any nonzero submodule
$V\subseteq Z^{\chi}(\l)$ contains the maximal vector $v$, hence
$V=Z^{\chi}(\l)$. So $Z^{\chi}(\l)$ is simple. \par
 Conversely, assume $Z^{\chi}(\lambda)$ is simple. We see that $w=:(\bar p-1,\dots,
 \bar p-1)^-v$ is a minimal vector in $Z^{\chi}(\lambda)$. i.e., $f_{\a_i}w=0$
 for every $\a_i\in \Phi^+$. Let $\g=\mathfrak N^-\oplus H\oplus\mathfrak  N^+$. Then we get
 $Z^{\chi}(\lambda)=u(\mathfrak N^+)w$. Since
 $\text{dim}Z^{\chi}(\lambda)=2^{\text{dim}\mathfrak N^+_{\bar
 1}}p^{\text{dim}\mathfrak N^+_{\bar 0}}$, the set $\{(i_1,i_2,\dots,
 i_n)w|0\leq i_j\leq \bar p-1,j=1,\dots,n\}$ is a basis of $Z^{\chi}(\lambda)$.
 Therefore,
 $$0\neq (\bar p-1,\dots,\bar p-1)w=(\bar p-1,\dots,\bar p-1)(\bar p-1,\dots,\bar p-1)^-v
 =f(\lambda)v.$$ Hence we get $f(\l)\neq 0$.
\end{proof}
Fixing a simple root system  $\Delta$ and  the positive root system
$\Phi^+$,  let $\mathfrak{N}^+=\oplus
_{\alpha\in{\Phi^+}}\g_{\alpha}$. For $\alpha\in \Delta_0$, the
refection $s_{\a}$ can be extended to an automorphism of $\g$, also
denoted by $s_{\a}$.  We have \cite[p.277]{ja} $$s_{\a}=\text{exp}
ade_{\a}\text{exp} ad(-f_{\a})\text{exp} ade_{\a}.$$ Assume $p>3$.
Then $s_{\a}$ is well defined in all cases. Each $s_{\a}$ can be
naturally extended to an automorphism of the tensor algebra $T(\g)$
such that $s_{\a}(1)=1$. Let $\mathfrak{I}$ be the two sided ideal
of $T(\g)$ generated by $$ \{x\otimes
y-(-1)^{\mathrm{p}(x)\mathrm{p}(y)}y\otimes
x-[x,y],z^p-z^{[p]}-\chi(z)^p\cdot 1|x, y\in \g_{\0}\cup\g_{\bar 1},
z\in\g_{\0}\}.$$ Then it is clear that $s_{\a}$ stabilizes
$\mathfrak{I}$. Hence $s_{\a}$ induces an automorphism of
$u(\g,\chi)=U(\g)/\mathfrak{I}$, also denoted by $s_{\a}$.\par

The induced action of $s_{\a}$ on $H^*$ is defined as:
$$
s_{\a}\cdot\lambda (h)=\lambda(s^{-1}_{\a} h), \text{for}\quad
\lambda\in H^*, h\in H.$$

 We identify $H$
with $(H^*)^*$ by letting $h(\lambda)=\lambda(h)$ for any
$\lambda\in H^*$ and $h\in H$. So $U(H)$ can be identified with the
polynomial function ring on $H^*$.  If $\text{dim}H=s$, regard $H^*$
as the affine variety $\mathbf F^s$. Then all weights related to
$\chi$ makes a closed subset $X$ defined by  the relation
$\l^p-\l=\chi$ and $u(H,\chi)$ is  the coordinate algebra. \par Put
elements of  $\Delta$ in a fixed order: $\a_1,\a_2,\dots, \a_r$. For
$\g\neq gl(m|n)$, we take a basis of $H$:
$h_1=:h_{\a_1},\dots,h_r=:h_{\a_r}$. For $\g=gl(m|n)$, fix a  $h$
such that $H=\{h_{\a}|\a\in \Delta\}\oplus \mathbf Fh$. We take the
basis of $H$: $h_1=:h_{\a_1},\dots,h_r=:h_{\a_r}, h_{r+1}=h$. Then
we have $X=\{(\l_1,\dots,\l_s)|\l_i^p-\l_i=\chi(h_i)^p\}\subseteq
F^s$ and $|X|=p^s$. For our purposes, it is sufficient to work on
$X$ instead of $H^*$.

For $f(h)=f(h_{\a_1},\dots,h_{\a_n})\in u(H,\chi)$, we have
$$s_{\a}f(h)(\lambda)=
s_{\a}f(h_{\a_1},\dots,h_{\a_n})(\lambda)$$$$=f(s_{\a}h_{\a_1},\dots,s_{\a}h_{\a_n})
(\lambda)$$$$=f(\lambda(s_{\a}h_{\a_1}),\dots,\lambda(s_{\a}h_{\a_n}))$$$$=
f(h_{\a_1},\dots,h_{\a_n})(s_{\a}^{-1}\lambda).$$   Let $v'$ be the
maximal vector defined by the Borel subalgebra
$s_{\a}(\mathfrak{B})$ of weight $\lambda\in H^*$ and character
$\chi$. i.e., $\mathbf Fv'$ is a 1-dimensional
$u(s_{\a}(\mathfrak{B}),\chi)$-module. Let $f^{\a}(\lambda)$
 be the polynomial defined with respect to the new positive root system
  $s_{\a}(\Phi^+)$.   Then we
have
$$
f^{\a}(\l)v'=s_{\a}(e_{\a_1})\dots
s_{\a}(e_{\a_n})s_{\a}(f_{\a_1})\dots s_{\a}(f_{\a_n})v'
$$
$$
=s_{\a}(e_{\a_1}\dots e_{\a_n}f_{\a_1}\dots f_{\a_n})v'.
$$
Assume $$e_{\a_1}\dots e_{\a_n}f_{\a_1}\dots f_{\a_n}=t(h)+\sum_i
u_in_i^+\in \pi(U(\g)_0),$$ where $u_i\in u(\g,\chi)$, $n_i^{+}\in
\mathfrak{N}^{+}$ and $t(h)=\gamma (e_{\a_1}\dots
e_{\a_n}f_{\a_1}\dots f_{\a_n})\in u(H,\chi).$\par We denote
$\mathfrak{N}^{'+}=s_{\a}(\mathfrak{N}^+)$. Then we get
$$s_{\a}( e_{\a_1}\dots e_{\a_n}f_{\a_1}\dots
f_{\a_n})=s_{\a}(t(h))+\sum_i u'_in_i^{'+},$$  where $u'_i\in
u(\g,\chi)$, $n_i^{'+}\in \mathfrak{N}^{'+}$. Thus, we have
$$f^{\a}(\l)v'=s_{\a}(t(h))v'=\lambda(s_{\a}(t(h)))v' =t(h)(s^{-1}_{\a}(\lambda))v'.$$
Since $t(h)(\lambda)=f(\lambda)$,  we get
$f^{\a}(\l)=f(s^{-1}_{\a}(\l))$.\par
\begin{lemma} For each $\a\in \Delta_0$, $f(\l)$ and $f(s_{\a}^{-1}(\l+\a))$
 have the same set of zeros. \end{lemma}
\begin{proof}Let $\Phi^+$ be the fixed positive root system with which
 $Z^{\chi}(\l)$ and $f(\l)$ are defined.  Let $v$ be the maximal
vector that defines $Z^{\chi}(\l)$. Assume $Z_1^{\chi}(\l)$ is the
baby verma module defined with respect to the new positive root
system $s_{\a}(\Phi^+)$. It is easy to see that
$v'=:f^{p-1}_{\a}v\in Z^{\chi}(\lambda)$ is a maximal vector with
respect to $s_{\a}(\mathfrak{B})$ and has weight $\l+\a$.
\par Since $v'$ is homogeneous, the inclusion map $\mathbf
Fv'\rightarrow Z^{\chi}(\l)$ induces a nonzero $\mathbf Z_2$-graded
$u(\g,\chi)$-homomorphism $\kappa$: $Z_1^{\chi}(\l)\longrightarrow
Z^{\chi}(\l)$. So $\text{Im}\kappa$ is a nonzero $\mathbf
Z_2$-graded submodule of $Z^{\chi}(\l)$.  Note that
$\text{dim}Z^{\chi}_1(\l)=\text{dim}Z^{\chi}(\l)$.\par If $f(\l)\neq
0$, then the simplicity of $Z^{\chi}(\l)$ shows that $\kappa$ is
onto, hence an isomorphism, so that $Z_1^{\chi}(\l)$ is simple and
hence $f^{\a}(\l)\neq 0$. On the other hand, if $f^{\a}(\l)\neq 0$,
then $Z^{\chi}_1(\l)$ is simple, hence the $\mathbf Z_2$-graded
submodule $\text{ker}\kappa$ must be zero. Therefore $\kappa$ is an
isomorphism. This implies that $f(\lambda)\neq 0$, since
$Z^{\chi}(\l)$ is simple.
\end{proof}

Let $\a\in \Delta$ be  a simple root. Consider the subalgebra  $A$
of $\g$ with basis $e_{\a}$,$f_{\a}$, $h_{\a}$. If $\a\in \Delta_0$,
then $A\cong sl_2$. If $\a\in \Delta_1$, then $A$ is a restricted
Heisenberg algebra of rank 1 (see \cite[p. 149]{s}). By the standard
results, for each semisimple character $\chi\in A^*$. A simple
$u(A,\chi)$-module of highest weight $a\in \mathbf F$  has dimension
less than $\bar{p}$ if and only if $a+1\in
\mathbf{F}_{\bar{p}}\setminus 0$.\par For each simple root $\a\in
\Delta$, we take a parabolic subalgebra of $\g$:
$\mathfrak{P}=\mathfrak{B}+\mathbf{F}f_{\a}$. Let
$\mathfrak{N}=\mathfrak{N}_{\bar{0}}\oplus \mathfrak{N}_{\bar{1}}$
be the nilradical of $\mathfrak{P}$. Let $\chi\in \g_{\bar 0}^*$ be
semisimple. For each $\l\in X$, a simple $u(A,\chi_{|A})$-module $V$
can be extended to a $u(\mathfrak{P},\chi)$-module as follows:
$$
e_{\beta}V=0, \beta\in \Phi ^+\setminus {\a}, hv=\l (h)v \text{ for
each}\quad  h\in H.
$$
Then $u(\g,\chi)\otimes _{u(\mathfrak{P},\chi)}V$  has dimension
$$2^{\text{dim}\mathfrak{N}_{\bar{1}}}p^{\text{dim}\mathfrak{N}_{\bar{0}}}\text{dim}V.$$
Obviously  $u(\g,\chi)\otimes _{u(\mathfrak{P},\chi)}V$ is an
epimorphic image of $Z^{\chi}(\l)$.  If $\l(h_{\a})+1\in \mathbf
F_{\bar p}\setminus 0$, so that $\text{dim}V<\bar{p}$, then we have
$\text{dim}u(\g,\chi)\otimes
_{u(\mathfrak{P},\chi)}V<\text{dim}Z^{\chi}(\l)$, so $Z^{\chi}(\l)$
is not simple, hence we have $f(\l)=0$.\par Recall that
$f(\l)=t(h)(\l)$. Let $\mathcal {V}(g)$ denote the set of zeros in
$X$ of $g\in u(H,\chi)$. Then we get
$\mathcal{V}(h_{\a}+1-c)\subseteq \mathcal{V}(t(h))$ for every
$c\in\mathbf F_{\bar p}\setminus 0$.\par By the earlier discussion,
each element in $u(H,\chi)$ is a polynomial of the variables
$h_1,\dots,h_r$ ($h_1,\dots, h_{r+1}$ for $\g=gl(m|n)$). It is no
loss of generality to assume $h_{\a}=h_1$. For every $c\in
\mathbf{F}_{\bar{p}}\setminus 0$, using the Remainder Theorem, we
have $$t(h)=(h_1+1-c)q(h)+r(h),\quad  q(h),r(h) \in u(H,\chi),$$
where the variable $h_1$ is not appearing in $r(h)$. Taking the
value at a point in $\mathcal{V}(h_{\a}+1-c)$, we get $r(h)=0$. Thus
$f(\l)=t(h)(\l)$ is divisible by the polynomial
$\l(h_{\a})+1-c$.\par For the basis of $H$ given earlier, we define
$\rho\in H^*$ such that $\rho(h_{i})=1$, $i=1,\dots,r$, and
$\rho(h_{r+1})=0$ if $\g=gl(m|n)$. Note that $\l+\rho\in X$ for
 any $\l\in X$.
\begin{corollary}(1) For each $\a\in \Delta_0$, $[(\l(h_{\a})+\rho (h_{\a}))^{{p}-1}-1]|f(\l)$.\par
(2) For each $\a\in \Delta_1$, $\l(h_{\a})|f(\l)$. \end{corollary}
\begin{lemma}
Let $\Phi=\Phi_0\cup \Phi_1$,
$\Phi^+=\Phi^+_0\cup\Phi_1^+,\Delta=\Delta_0\cup\Delta_1$ be the
root, positive root, simple root system of $\g$ respectively. Let
$W_0$ be the weyl group of $\g_{\0}$. Then for each $\beta\in
\Phi_1$, there is $\tau\in W_0$  such that $\tau (\beta)\in
\Delta_1$.
\end{lemma}
\begin{proof}
Let $\g_{\0}=\g_{\bar 0}^1\oplus \g_{\bar 0}^2$ as given in Section
2. Then $\Delta_0$ consist of simple roots of $\g^1_{\bar
0}$(denoted in terms of $\e_i$'s in \cite[p.51-53]{k1}) and those of
$\g^2_{\bar 0}$(denoted in terms of $\delta_i$'s in
\cite[p.51-53]{k1}). By definition, it is easy to get
$s_{\a}(\lambda)=\l-\l(h_{\a})\a$ for $\l\in H^*$ and $\a\in
\Delta_0$. Using the root system given in \cite[p.51-53]{k1}, one
sees that there is $\tau\in W_0$ such that $\tau (\beta)\in
\Delta_1$.
\end{proof}

\begin{theorem}$$
f(\lambda)=\pm\Pi _{\a\in \Phi^+_0}[(\l(h_{\a})+\rho
(h_{\a}))^{{p}-1}-1]\Pi_{\a\in \Phi^+_1}[\l(h_{\a})+\rho
(h_{\a})-1].
$$
\end{theorem}
\begin{proof} By Corollary 3.5,$$
\Pi _{\a\in \Delta_0}[(\l(h_{\a})+\rho
(h_{\a}))^{{p}-1}-1]\Pi_{\a\in \Delta_1}[\l(h_{\a})+\rho
(h_{\a})-1]|f(\l).
$$
 Then the theorem follows from  Lemma 3.6 and  a similar argument
  as that for \cite[Prop. 8]{r}.
\end{proof}

\section{The simplicity of the graded baby verma modules for
$\g=gl(m|n)$} Let $\g=gl(m|n)$. Then $\g=\g_{-1}\oplus \g_{\0}\oplus
\g_{1}$, where $$\g_{-1}=\oplus _{j>m,i\leq m}\mathbf Fe_{ji}\quad
\text{and}\quad \g_1=\oplus _{j>m,i\leq m}\mathbf Fe_{ij}.$$ The
reader may refer to \cite{k1} for more detailed description of
$\g_{-1}$ and $\g_{1}$.\par

 Let $\chi \in \g^*_{\bar{0}}$ and assume  $\chi(\mathfrak{N}_0^+)=0$, where
 $\mathfrak{N}^+_0=\oplus \sum_{i<j}\mathbf{F}e_{ij}\subseteq \g_{\bar{0}}$. Denote
$\g^+=\g_0+\g_1$. Let $M$ be a simple $u(\g_{\bar{0}},\chi)$-module.
We regard $M$ as a $u(\g^+,\chi)$-module by letting $\g_1M=0$. As in
\cite{zl}, we define the induced module
$Z^{\chi}(M)=u(\g,\chi)\otimes _{u(\g^+,\chi)}M$. We call it the
graded baby verma module.
\begin{lemma}Any simple $u(\g,\chi)$-module $\mathfrak{M}=\mathfrak{M}_{\0}\oplus\mathfrak{M}_{\bar{1}}$
is a quotient of $Z^{\chi}(M)$ for some simple
$u(\g_{\0},\chi)$-submodule $M\subseteq \mathfrak{M}_{\0}$ or
$M\subseteq\mathfrak{M}_{\bar 1}$.
\end{lemma}
\begin{proof} By \cite[Lemma 3.12]{zl}, there is
a simple $u(\g^+,\chi)$-submodule $M\subseteq \mathfrak{M}_{\0}$ or
$M\subseteq\mathfrak{M}_{\bar 1}$. Then the inclusion map
$M\rightarrow \mathfrak M$ induces a $\mathbf Z_2$-graded
$u(\g,\chi)$-module homomorphism $f: Z^{\chi}(M)\rightarrow
\mathfrak{M}$. Since $\mathfrak{M}$ is simple $f$ is an epimorphism.
\end{proof}
 Let $v \in M$ be a maximal
vector of weight $\l$. Let $e_1,\dots, e_l$ (resp., $f_1,\dots,f_l$)
denote the following basis of $\g_1$(resp., $\g_{-1}$):
$$\{e_{ij}|i\leq m,j>m\}(resp., \{e_{ij}|i>m,j<m\}).$$ We put them
in the order of ascending heights. For $e_1\dots e_lf_1\dots f_l\in
\pi(U(\g)_0)$, assume $e_1\dots e_lf_1\dots f_lv=f_1(\lambda)v$.

\begin{proposition}
(1) If $f_1(\l)\neq 0$, then  $Z^{\chi}(M)$ is simple.\par (2) If
 $Z^{\chi}(M)$ is simple, then
$f_1(\lambda)\neq 0$.
\end{proposition}
\begin{proof}
(1) Let $V$ be any proper submodule of $Z^{\chi}(M)$. Taking $$0\neq
\sum f_{i_1}\dots f_{i_t} m_{i_1,\dots,i_t}\in V,\quad
m_{i_1,\dots,i_t}\in M,$$ let $t$ be the smallest such that
$f_{i_1}\dots f_{i_t} m_{i_1,\dots,i_t}\neq 0$. Let
$$\{j_1,\dots,j_{l-t}\}=\{1,\dots,l\}-\{i_1,\dots,i_t\}.$$ By
applying $f_{j_1}\dots f_{j_{l-t}}$, we  get $0\neq f_1\dots f_l
m\in V$ for some $0\neq m\in M$. Hence
$$
e_1\dots e_lf_1\dots f_l m\in V.
$$
 Let
$\g_{\bar{0}}=H\oplus \mathfrak{N}^-_{0}\oplus \mathfrak{N}^+_{0}$.
Since $M$ is a simple $u(\g_{\0},\chi)$-module, there is a
polynomial
$$f=\sum n^-_ih_in^+_i\in u(\g_{\0},\chi)=u(\mathfrak{N}^-_0,\chi)
u(H,\chi)u(\mathfrak{N}^+_0,\chi)$$ such that $fm$ is the maximal
vector $v\in M$. Note that for $i,j\leq m$ with $i\neq j$,
$$[e_{ij}, e_1\dots e_l]=[e_{ij}, f_1\dots f_l]=0, [h,e_1\dots e_lf_1\dots f_l]=0
\quad\text{for all}\quad h\in H.$$ So we get
$$[f,e_1\dots e_lf_1\dots
f_l]=[\Sigma n^-_ih_in^+_i, e_1\dots e_lf_1\dots f_l]=0.$$ This
shows that $e_1\dots e_lf_1\dots f_l v\in V$ and hence $v\in V$. So
(1) follows immediately.\par (2) Let $v_1=v,v_2,\dots,v_s$ be a
basis of $M$. Since $\g_1M=0$ and
$\text{dim}Z^{\chi}(M)=2^{\text{dim}\g_{-1}}\text{dim}M$. $$
\{f_{i_1}\dots f_{i_k}v_i|0\leq i_1<\dots <i_k\leq l,
i=1,\dots,s\}$$ is a basis of $Z^{\chi}(M)$. Let $M'=f_1\dots f_lM$,
then $\text{dim}M'=\text{dim}M=s$. \par Since $\g_{-1}M'=0$,
$$Z^{\chi}(M)=u(\g_1)M'=\langle e_{i_1}\dots e_{i_k}f_1\dots f_lv_i|
0\leq i_1<\dots<i_k\leq l, i=1,\dots, s\rangle.$$ Therefore,
$e_1\dots e_lf_1\dots f_l v\neq 0$, so that $f_1(\l)\neq 0$.
\end{proof}
Denote $f_0(\l)=\Pi_{\a\in \Phi^+_0}[(\l(h_{\a})+\rho
(h_{\a}))^{p-1}-1].$

\begin{lemma}Let $\chi$ be semisimple. Then $f_1(\l)=c
f(\l)/f_0(\l)$, $0\neq c\in\mathbf F$.
\end{lemma}
\begin{proof}If $\l$ is
such that $f_0(\l)\neq 0$ but $f(\l)=0$, then by \cite[Th. 3]{r},
$M(\l)$ is simple, where $M(\lambda)$ is the  simple
$u(\g_{\0},\chi)$-module with the unique maximal vector $v$ of
weight $\lambda$. Since $Z^{\chi}(M(\l))$ is not simple, we get
$f_1(\l)=0$.  By Theorem 3.7, $f(\l)/f_0(\l)\in u(H,\chi)$ is the
product of linear factors. So we have $[f(\l)/f_0(\l)]|f_1(\l)$. A
direct computation on $e_1\cdots e_lf_1\cdots f_lv$ shows that the
highest term of $\gamma(e_1\cdots e_lf_1\cdots f_l)$ is
$h_{\beta_1}\cdots h_{\beta_l}$, where $\beta_1,\dots,\beta_l$ are
all the positive odd roots of $\g$. So we get
$\text{deg}(f_1(\l))=l$. That is
$$\text{deg}(f_1(\l))=\text{deg}(f(\l))-\text{deg}(f_0(\l)).$$ Hence we get
$f_1(\l)=c f(\l)/f_0(\l)$ for some $c\neq 0$.\end{proof} By Theorem
3.7,  $f_1(\l)=\Pi_{\a\in\Phi^+_1}((\l+\rho)(h_{\a})-1)$(up to a
constant multiple). For $\a=\e_i-\delta_j\in \Phi_1^+$, it is easy
to see that $h_{\a}= e_{ii}+e_{i+j,i+j}$. So we get
$\rho(h_{\a})=m-i+j\in \mathbf F_p$. Let $\chi$ be semisimple. If
$\chi(h_{\a})\neq 0$, then $\l(h_{\a})\notin \mathbf{F}_p$.  So we
have

\begin{theorem} Let $\g=gl(m|n)$ and let $\chi$ be semisimple. If
$\chi(h_\a)\neq 0$ for every $\a\in \Phi^+_1$, then $Z^{\chi}(M)$ is
simple.\end{theorem}  Recall $Z^{\chi}(\lambda)=Z_{\0}\oplus Z_{\bar
1}$. Let $v\in Z_0$ be the maximal vector that defines
$Z^{\chi}(\lambda)$. Let $V_0:=u(\g_{\0},\chi)v\subseteq Z_{\0}$.
Since $[e_{ij}, \g_1]\subseteq \g_1$ and $[e_{ij},\g_{-1}]\subseteq
\g_{-1}$ for any $e_{ij}\in \g_{\0}$, we have $\g_1 V_0=0$. Let
$$V_0\supseteq V_1\supseteq \dots \supseteq V_s\supseteq 0$$ be a
composition series of $u(\g_{\0},\chi)$-submodules. Then it can also
be regarded as a composition series of $u(\g^+,\chi)$-modules. This
induces a filtration of $\mathbf Z_2$-graded $u(\g,\chi)$-submodules
of $Z^{\chi}(\lambda)$:
$$Z^{\chi}(\lambda)=u(\g,\chi)\otimes _{u(\g^+,\chi)}V_0\supseteq
u(\g,\chi)\otimes _{u(\g^+,\chi)}V_1\supseteq \dots \supseteq
u(\g,\chi)\otimes _{u(\g^+,\chi)}V_s\supseteq 0.$$ We see that each
factor module $$u(\g,\chi)\otimes
_{u(\g^+,\chi)}V_{i-1}/u(\g,\chi)\otimes V_i\cong u(\g,\chi)\otimes
_{u(\g^+,\chi)}V_{i-1}/V_i=Z^{\chi}(V_{i-1}/V_i), i=1,\dots, s$$ is
a graded baby verma module.
\begin{corollary} Let $\g=gl(m|n)$. If $f_1(\lambda)\neq 0$, then
the baby verma module $Z^{\chi}(\lambda)$ has a composition series
induced by that of the $u(\g_{\0},\chi)$ baby verma module
$u(\g_{\0},\chi)v$. In addition, the loewy length equals one if and
only if $f_0(\lambda)\neq 0$.
\end{corollary}
\section{$\g=gl(m|n)$}

\subsection{The Frobenius superalgebra}
\begin{definition}\cite{k1} Let $V=V_{\0}\oplus V_{\bar{1}}$ be a
$\mathbf{Z}_2$-graded space and $f$ is a bilinear form on $V$. Then
$f$ is called supersymmetric if
$f(a,b)=(-1)^{\mathrm{p}(a)\mathrm{p}(b)}f(b,a)$ for any $a,b\in
V_{\0}\cup V_{\bar{1}}$.
\end{definition}
\begin{definition}An associative $\mathbf{F}$-superalgebra $A$ is called a
Frobenius superalgebra if it has a non-degenerate invariant bilinear
form $f$. A Frobenius superalgebra is said to be supersymmetric if
$f$ is supersymmetric.
\end{definition}
Let $A=A_{\bar 0}\oplus A_{\bar 1}$ be a Frobenius superalgebra.  We
use the notation $A_L$ (resp., $A_R$) to denote the left (resp.,
right) regular $A$-module $A$. Let $$A_{\bar i}^*=:\{f\in
A^*|f(A_{\overline {i+1}})=0\}, \bar i\in \mathbf Z_2.$$ Then
$A^*=A^*_{\bar 0}\oplus A^*_{\bar 1}$. It is easy to check that
$A_L^*$ (resp., $A^*_R$) is a right (resp., left) $\mathbf
Z_2$-graded $A$-module.
\begin{lemma}
Let $A$ be a finite dimensional superalgebra. Then the following are
equivalent:\par (1) $A$ is Frobenious.\par (2) $A_L\cong (A_R)^*$,
$A_R\cong (A_L)^*$.
\end{lemma}
The proof is similar to that of \cite[Th. 61.3, p.414]{c}.

 If we take
the polynomial ring $\mathcal {O}$ generated by $\{x^p-x^{[p]}|x\in
\g_{\bar{0}}\}$, then $\mathcal{O}$ is in the center of $U(\g)$. By
applying a similar argument as that for \cite[Prop.1.2]{f},
 one gets
\begin{proposition}Let $\g$ be a finitely dimensional restricted Lie
superalgebra. For any $\chi\in \g_{\bar{0}}$(considered as a linear
function on $\g$ by letting $\chi(\g_{\bar{1}})=0$),  $u(\g,\chi)$
is a Frobenius superalgebra. Moreover, $u(\g,\chi)$ is
supersymmetric if $str(ad x)=0$ for any $x\in \g$.
\end{proposition}
\begin{definition}Let $\g$ be a restricted Lie superalgebra. We call
$\g$ unipotent if for any $x\in \g_{\0}$, there is $r>0$ such that
$x^{[p]^r}=0$.
\end{definition}

By Lemma 2.1(2), we get
\begin{corollary}If $\g$ is unipotent, then the trivial $\g$-module $\mathbf{F}$
 is the only simple $u(\g)$-module.
\end{corollary}
\begin{lemma} If $\g$ is unipotent, then each $u(\g,\chi)$ has only
one simple module (up to isomorphism).
\end{lemma}
\begin{proof}Let $M=M_{\0}\oplus M_{\bar 1}$  and $M'=M'_{\bar 0}\oplus M'_{\bar 1}$
be two simple $u(\g,\chi)$-modules. For $\bar i, \bar j\in \mathbf
Z_2$,  let $$\text{Hom}_{\mathbf F}(M,M')_{\bar i}=\{f\in
\text{Hom}_{\mathbf F}(M, M')| f(M_{\bar{j}})\subseteq M'_{\overline
{i+j}}\}.$$ Then
$$\text{Hom}_{\mathbf{F}}(M,M')=\text{Hom}_{\mathbf{F}}(M,M')_{\0}\oplus
\text{Hom}_{\mathbf{F}}(M,M')_{\bar{1}}$$ is a $u(\g)$-module. By
Corollary 5.6, $\text{Hom}_{\mathbf{F}}(M,M')_{\0}$ contains a
trivial submodule $\mathbf Fx$. Suppose $x=x_{\bar 0}+x_{\bar 1}$,
$x_{\bar i}\in \text{Hom}_{\mathbf{F}}(M,M')_{\bar i}$, $\bar i\in
\mathbf Z_2$. Then both $\mathbf Fx_{\bar i}$ are trivial
$u(\g)$-submodules. So we may assume $x$ is homogeneous. Therefore
$$\text{Hom}_{\g}(M,M')_{\0}\cup \text{Hom}_{\g}(M,M')_{\bar 1}\neq
0.$$ Take $$0\neq x \in \text{Hom}_{\g}(M,M')_{\0}\cup
\text{Hom}_{\g}(M,M')_{\bar 1},$$ then $x(M_{\bar 0})\oplus
x(M_{\bar 1})$ is a nonzero $\mathbf Z_2$-graded submodule of $M'$.
Since $M$ and $M'$ are simple,  $M\cong M'$.
\end{proof}
\begin{lemma}Let $\g$ be a finite dimensional unipotent Lie
superalgebra. Then the regular(both left and right) $u(\g)$-module
$u(\g)$ has  a unique trivial submodule $\mathbf Fv $. Therefore,
$v\in u(\g)_{\bar 0}\cup u(\g)_{\bar 1}$.
\end{lemma}
\begin{proof} We only give the proof for the left regular $u(\g)$-module.
By Corollary 5.6, there is $v\in u(\g)$ such that $\mathbf Fv$ is a
trivial $u(\g)$-submodule. Let $f$ be the non-degenerate invariant
bilinear form on $u(\g)$. By assumption, $f(x, v)=f(1, xv)=0$ for
any $x\in \g$. So $v$ is in the (right ) orthogonal complement of
$u(\g)\g$. Since $u(\g)\g$ has codimension 1 in $u(\g)$, its
orthogonal complement is 1-dimensional. This implies that $u(\g)$
has a unique trivial submodule $\mathbf Fv$. The uniqueness implies
that $v$ is homogeneous.
\end{proof}
\begin{lemma}Let $\g$ be a finite dimensional unipotent Lie
superalgebra. If $str(ad x)=0$ for all $x\in \g$. Let $\mathbf
Fv_L$(resp., $\mathbf Fv_R$) be  the unique trivial
$u(\g)$-submodule of $u(\g)_L$(resp., $u(\g)_R$). Then $\mathbf
Fv_L=\mathbf Fv_R$.
\end{lemma}
\begin{proof} By assumption, $u(\g)$ is a supersymmetric superalgebra. Let $f$ be
the non-degenerate invariant bilinear form on $u(\g)$. The super
symmetry of $f$ implies that both $v_L$ and $v_R$ are in the right
orthogonal complement of $u(\g)\g$, which is 1-dimensional, so
$v_L=v_R$(up to a nonzero constant  multiple).\end{proof}
\begin{definition} \cite[p.15]{k1} A derivation of degree $s\in
\mathbf Z_2$ of a superalgebra $A$ is an endomorphism $D\in End_sA$
with property
$$D(ab)=D(a)b+(-1)^{s\mathrm{p}(a)}aD(b).$$\end{definition}
Let $\g$ be a Lie superalgebra. By the Jacobi identity, $ada$:
$b\rightarrow [a,b]$ is a derivation of $\g$.\par Assume $\g_{\bar
1}$ has a basis $u_1,\dots,u_l$. Then $U(\g)=U(\g)_{\0}\oplus
U(\g)_{\bar 1}$ is naturally $\mathbf Z_2$-graded, where
$$U(\g)_{\0}=\sum_{\overset{i_1<\cdots <i_s}{s \quad\text{is even}}}u_{i_1}\dots
u_{i_s}U(\g_{\0}), U(\g)_{\bar 1}=\sum_{\overset{i_1<\cdots <i_s}{s
\quad\text{is odd}}} u_{i_1}\dots u_{i_s}U(\g_{\0}).$$ Generally,
taking $x_1x_2\cdots x_k\in U(\g)$,
$x_1,\dots,x_k\in\g_{\0}\cup\g_{\bar 1}$, it is easy to see that
$\mathrm{p}(x_1\cdots x_k)=\sum^k_{i=1}\mathrm{p}(x_i)$.\par Let
$a\in \g_{\0}\cup\g_{\bar 1}$. We now extend $ada$ to a derivation
of $U(\g)$. Let $f\in U(\g)_{\theta}$, $\theta\in\mathbf Z_2$. We
define $$ ada (f)=af-(-1)^{\mathrm{p}(a)\mathrm{p}(f)}fa.$$ Use
induction we can show that
\begin{lemma}Let $f=x_1\cdots x_k\in U(\g)_{\theta}$, where $x_i\in\g_{\0}\cup\g_{\bar 1}$.
Then $$ada(f)=\sum (-1)^{\mathrm{p}(a)\mathrm{p}(x_1\cdots
x_{i-1})}x_1\cdots [a,x_i]x_{i+1}\cdots x_k.$$
\end{lemma}Recall $I_{\chi}$ for $\chi\in \g^*$ given in Section 2.
Let $\g$ be a restricted Lie superalgebra and let $a\in
\g_{\0}\cup\g_{\bar 1}$. For each generator
$x^p-x^{[p]}-\chi(x)^p\cdot 1\in U(\g)_{\0}$, it is easy to see that
$ada (x^p-x^{[p]}-\chi(x)^p\cdot 1)=0$. This induces a derivation on
$u(\g,\chi)$. We denote it also by $ada$. Note that $I_{\chi}$ is
generated by elements in $U(\g)_{\0}$. Then $u(\g,\chi)$ inherits
the $Z_2-gradation$ from that of $U(\g)$, and Lemma 5.11 also holds
in $u(\g,\chi)$.
\subsection{The Kac-weisfeiler Theorem }
The $\mathbf F$-vector space $gl(m+n)$ has a structure of the Lie
superalgebra  $\g=gl(m|n)$, and also a structure of   Lie algebra
$\mathrm{gl}(m+n)$. The two strutures share the same cartan
subalgebra $H=\langle e_{ii}|i=1,\dots,m+n\rangle$ and the
corresponding root space decomposition.\par
 Let $\g=gl(m|n)$. Then $\g_{\0}=gl(m)\oplus gl(n)$
 and $\g=H+\oplus \g_{\a}$, where $\g_{\a}=\mathbf Fe_{ij}$,
$1\leq i,j\leq m+n$. With respect to $H$, the root system of $\g$ is
given in \cite[p.51]{k1}. \par If we write $\e_{m+1}=\delta_1$,
\dots, $\e_{m+n}=\delta_n$, where $\e_i(e_{jj})=\delta_{ij}$, then
the root system of $gl(m|n)$   coincides with that of
$\mathrm{gl}(m+n)$(denoted by $\Phi$), and each simple root system
of $gl(m|n)$ coincides with  that of $\mathrm{gl}(m+n)$.\par  For
$\a\in \Phi^+$, we use $e_{\a}$(resp., $f_{\a}$) to denote the
positive(resp., negative) root vector. i.e., $\g_{\a}=\mathbf
Fe_{\a}$(resp., $\g_{-\a}=\mathbf Ff_{\a}$).\par If we use $[,]_s$
momentarily to denote the Lie product in $gl(m|n)$, then we see that
$$ (*)\quad\quad [e_{\a_i}, e_{\a_j}]_s=\pm [e_{\a_i}, e_{\a_j}],
[f_{\a_i},f_{\a_j}]_s=\pm[f_{\a_i},f_{\a_j}],$$$$
[e_{\a_i},f_{\a_j}]_s=\pm [e_{\a_i},f_{\a_j}]\quad\text{for}\quad
\a_i\neq \a_j.$$ For each positive root $\a\in \Phi^+_{1}$, we
define $s_{\a}(\beta)=:\tau_{\a}(\beta), \beta \in \Phi$, where
$\tau_{\a}$ is the reflection defined in the root system of
$\mathrm{gl}(m+n)$. The definition of $s_{\a}$ is justified by the
next lemma.
\begin{lemma} $s_{\a}(\Phi^+)$ is a positive root system of
$gl(m|n)$ with the simple root system
$s_{\a}(\Delta)$.\end{lemma}\begin{proof}  Since $\tau_{\a}(\Phi^+)$
is a positive root system of $\mathrm{gl}(m+n)$ with simple roots
$\tau_{\a}(\Delta)$, $\tau_{\a}(\Delta)$ is a minimal subset of
$\tau_{\a}(\Phi^+)$ satisfying \par  (1) $\{e_{\gamma}|\gamma\in
\tau_{\a}(\Delta)\}$ generates $\{e_{\gamma}|\gamma\in
\tau_{\a}(\Phi^+)\}$, $\{f_{\gamma}|\gamma\in \tau_{\a}(\Delta)\}$
generates $\{f_{\gamma}|\gamma\in \tau_{\a}(\Phi^+)\}.$\par (2)
$\{e_{\gamma},f_{\gamma}|\gamma\in \tau_{\a}(\Delta)\}$ generates
$\mathrm{gl}(m+n)$.\par (3)
$[e_{\gamma},f_{\beta}]=\delta_{\gamma,\beta}h\in H$, $\gamma,\beta
\in \tau_{\a}(\Delta)$.\par Note that
$$(**)\quad [h,e_{\gamma}]_s=[h,e_{\gamma}]=\gamma(h)e_{\gamma}, \quad\text{for
every}\quad h\in H.$$ So each $e_{\gamma}$ above is also a root
vector in the Lie superalgebra $gl(m|n)$. Then  $(*)$ says that
$\{e_{\gamma}|\gamma\in \tau_{\a}(\Delta)\}$ is the set of root
vectors  of a simple root system $\tilde{\Delta}$ of $\g=gl(m|n)$,
and $\{e_{\gamma}|\gamma\in \tau_{\a}(\Phi^+)\}$ is the set of root
vectors of the corresponding positive root system
$\tilde{\Phi}^+$.\par By $(**)$, we have
$\tilde{\Delta}=\tau_{\a}(\Delta)$ and
$\tilde{\Phi}^+=\tau_{\a}(\Phi^+)$; that is, $\tau_{\a}(\Phi^+)$ is
a positive root system of $gl(m|n)$ with the simple root system
$\tau_{\a}(\Delta)$.\end{proof}\begin{lemma} Let $I=I_{\0}\oplus
I_{\bar 1}$ be a restricted ideal of a subalgebra $\g'$ of the Lie
superalgebra $\g$. Assume $\chi(I)=0$ and $u(I)\subseteq
u(\g',\chi)$ be the reduced enveloping algebra of $I$. Let $\mathbf
F v_L$ be the unique 1-dimensional trivial submodule of $u(I)$.
Suppose $v_L=x_1^{i_1}\dots x^{i_m}_m$, where $x^{i_k}_k\in I_{i_k}$
for each $i_k\in \mathbf Z_2$. Then for every homogeneous $a\in
\g'$, $ad a(v_L)=\l v_L$ for some $\l\in \mathbf F$. If $ad a$ is
nilpotent, then $\l=0$.
\end{lemma}
\begin{proof} Let $x\in I$ be homogeneous.
Then
$$xada(v_L)=x(av_L-(-1)^{\mathrm{p}(v_L)}v_La)$$$$=[x,a]v_L+(-1)^{\mathrm{p}(x)\mathrm{p}
(a)}axv_L-(-1)^{\mathrm{p}(v_L)}xv_La=0.$$ Therefore $ada(v_L)=\l
v_L$ for some $\l\in \mathbf F$.
\end{proof}
Note that if $\g=\g_{\0}$, the lemma recovers \cite[Lemma
8.3]{f}.\par Let $\chi \in \g^*_{\0}$.
 Without loss of generality we assume
$\chi$
 vanishes on ${N}^+=\sum_{\a>0} \g_{\a}$. Let
$\chi=\chi_s+\chi_n$, where $\chi_s$ is the semisimple part and
$\chi_n$ is the nilpotent part. Let
$$\mathfrak{l}'=C_{\g}(\chi_s)=\{x\in \g|\chi_s([x,-])=0\}.$$ Then we
get
$$
\mathfrak{l}'=H\oplus\sum_{\chi(h_{\a})=0} \g_{\a}.
$$
Let $\mathfrak{l}=[\mathfrak{l}',\mathfrak{l}']$ and let
$\mathfrak{B}=H+N^+$. $\mathfrak{P}=\mathfrak{B}+\mathfrak{l}$ is a
parabolic subalgebra of $\g$. First we take the natural root system
of $\g$ with the simple roots
$$
\Delta=\{\e_1-\e_2,\dots, \e_m-\d_1,\dots,\d_{n-1}-\d_n\}.
$$
 Then $\mathfrak{N}=\sum_{\chi(h_{\a})\neq 0,\a>0}\g_{\a}$ is the
nilradical of $\mathfrak{P}$. Let
$\mathfrak{N}=\mathfrak{N}_{\0}\oplus \mathfrak{N}_{\bar{1}}$, where
$$\mathfrak{N}_{\bar i}=\sum_{\chi(h_{\a})\neq 0,\a\in \Phi^+_{\bar i}}\g_{\a}, \bar{i}\in \mathbf Z_2.$$
Let $\Phi_{\mathfrak{l}}$ be the root system of $\mathfrak{l}$.
Denote $\Phi'=\{\a\in \Phi^+|\chi(h_\a)\neq 0\}$. By \cite[p.
140]{k3}, $\chi(e_{\a})=0$ for every $\a\in \Phi^+\cup -\Phi'$.
\par

\begin{lemma}({\cite{h},\cite{f}}) There is an order of the roots in $\Phi'$:
$\Phi'=\{\a_1,\a_2,\dots, \a_s\}$ such that if $\Phi^+_1=\Phi^+$,
$\Delta_1=\Delta$. $\Phi^+_{i+1}=s_{\a_i}(\Phi^+_i)$ and
$\Delta_{i+1}=s_{\a_i}(\Delta_i)$. Then \par (1) $\Phi^+_i$ is a
system of positive roots in $\Phi$ with simple roots $\Delta_i$ and
$\a_i\in \Delta_i$.\par (2) For each $i\leq s$,
$\{-\a_1,-\a_2,\dots, -\a_i\}$ is a closed subsystem normalized by
$\Phi_{\mathfrak l}^+$.
\end{lemma}\begin{proof}
Now that $s_{\a}$ is defined for $\a\in \Phi_1$,  the proof in
 \cite{h} and \cite{f} also works here.  We give here only a
$\chi$-related proof to the first part.  We treat $\chi$ as a linear
function of $\g$ with $\chi(\g_{\bar{1}})=0$. Note that we can
assume $\g=sl(m|n)$ for this lemma.\par

Suppose $\Delta\cap \Phi'=\phi$. Then we have $\Delta\subseteq
\Phi_{\mathfrak{l}}$. Since $\Phi_{\mathfrak{l}}$ is a closed root
system and $H\subseteq [\g,\g]$, we get $\chi(H)=0$. Thus, $\chi$ is
nilpotent and the case being trivial.
\par Taking $\a_1\in \Delta\cap\Phi'$, we let
$\Delta_2=s_{\a_1}(\Delta_1)$ and $\Phi^+_2=s_{\a_1}(\Phi^+_1)$.
Note that $s_{\a_i}(\Phi')\subseteq \Phi'$. Using a similar argument
as above, we can find $\a_2\in \Delta_2\cap \Phi'$. We continue this
process.  Suppose there is $i<j$ such that $\a_i=\a_j$. Then since
$-\a_i\in \Phi^+_j$ and  $ \a_j\in \Delta_j\in \Phi^+_j$, we get a
contradiction. Therefore $\Phi'$ can be ordered as required.
\end{proof}
 Let
$\Phi'=\{\a_1,\dots,\a_s\}$ be in the order as above. As in
\cite{f}, we let $F_i=\langle f_{\a_1},\dots,f_{\a_i}\rangle$. Then
since $\{\a_1,\dots,\a_i\}$ is closed,  $F_i$ is a subalgebra of
$\g$. Since $\chi(F_i)=0$, $u(F_i)=u(F_i,\chi)\subseteq u(\g,\chi)$.

\begin{theorem}
Let $M=M_{\0}\oplus M_{\bar 1}$ be a simple $u(\g,\chi)$-module and
$M'\subseteq M$ a $\mathbf Z_2$-graded simple
$u(\mathfrak{P},\chi)$-submodule. Then $M\cong u(\g,\chi)\otimes
_{u(\mathfrak{P},\chi)}M'$. In particular, we have $$dim M=p^{dim
\mathfrak{N}_{\0}}2^{dim \mathfrak{N}_{\bar{1}}}dim M'.$$
\end{theorem}
\begin{proof}Let $\mathfrak{B}_i$ denote the Borel subalgebra of $\g$ with
respect to $\Delta_i$ and
$\mathfrak{P}_i=\mathfrak{B}_i+\mathbf{F}f_{\a_i}$. Denote
$\mathfrak{N}_i=\oplus _{\a\in \Phi^+_i,\a\neq \a_i}\g_{\a}$. Then
$\mathfrak{N}_i$ is the nilradical of $\mathfrak{P}_i$. Since $M$ is
simple, $M$ is the quotient of $W=u(\g,\chi)\otimes
_{u(\mathfrak{P},\chi)}M'$. Note that $\mathfrak{B}_{i+1}\subset
\mathfrak{P}_i$. \par  We define an increasing filtration of $W$:
$$W_1=1\otimes M', W_{i+1}=u(\mathfrak{P}_i,\chi)\otimes
_{u(\mathfrak{B}_i,\chi)}W_i, i=1,\dots,s.$$ Clearly we have
$W_{s+1}=W$.\par
 Let $W'$ be a simple $u(\g,\chi)$-submodule of $W$ and choose $i$ minimal such that $W'\cap
W_{i+1}\neq 0$.\par  Suppose $i>0$. We shall derive a contradiction.
 We first show that
$$(1)\quad W_{i+1}^{\mathfrak{N}_i}=(u(\mathfrak{P}_i,\chi)\otimes
_{u(\mathfrak{B}_i,\chi)}W_i)^{\mathfrak{N}_i}=u(\mathfrak{P}_i,\chi)\otimes
_{u(\mathfrak{B}_i,\chi)}W_i^{\mathfrak{N}_i}.$$ For each root $\a$
of $ \mathfrak{N}_i$, we have $[e_{\a},f_{\a_i}]\in \mathfrak{N}_i$,
since $\mathfrak{N}_i$ is the nilradical of $\mathfrak{P}_i$. Let
$$x=\sum_{0\leq j\leq \bar{p}-1}f^j_{\a_i}\otimes x_j\in
W^{\mathfrak{N}_i}_{i+1}.$$  Then we obtain $f_{\a_i}^k x\in
W^{\mathfrak{N}_i}_{i+1}$ for every $0\leq k\leq \bar p-1$.\par
Assume $\a_i\in \Phi^+_0$. Since $\chi(f_{\a_i})=0$, so that
$f^p_{\a_i}=0$ in $u(\g,\chi)$, we have
$$f^k_{\a_j}x= \sum_{0\leq j\leq  p-1-k}f^{j+k}_{\a_i}\otimes x_j, 0\leq k\leq p-1.$$
Assume $\a_i\in \Phi^+_1$. Since $f^2_{\a_i}=0$ in $u(\g,\chi)$(
also in $U(\g)$), we have $$f^k_{\a_j}x= \sum_{0\leq j\leq
1-k}f^{j+k}_{\a_i}\otimes x_j$$ for $k=0,1$. \par Hence we have, for
any $\a_i\in \Phi^+$,
$$f^k_{\a_j}x= \sum_{0\leq j\leq  \bar p-1-k}f^{j+k}_{\a_i}\otimes x_j$$
for $0\leq k\leq \bar p-1$, and $f^{\bar p}_{\a_i} x=0$.\par
  We use induction
on $k$ to prove that $x_k\in W_i^{\mathfrak{N}_i}$. For $k=0$, take
an arbitrary $y\in \mathfrak{N}_i$.  Since $[f_{\a_i},y]\in
\mathfrak{N}_i\subseteq \mathfrak{B}_i$, we get
$$0=yf^{\bar p-1}_{\a_i}x=yf^{\bar p-1}_{\a_i}\otimes x_0=f^{\bar p-1}_{\a_i}\otimes
yx_0+\sum_{l<\bar p-1} f^l_{\a_i}\otimes w'_l, w'_l\in W_i.$$ This
gives us $y x_0=0$ and hence $x_0\in W_i^{\mathfrak{N}_i}$.
\par Assume  $x_k\in W_i^{\mathfrak{N}_i}$ for all $k<m\leq \bar
p-1$, and consider the case  $k=m$. For any $y\in \mathfrak{N}_i$,
we have
$$0=yf^{\bar p-1-m}_{\a_i}x=y\sum_{0\leq k\leq m}f^{\bar
p-1-m+k}_{\a_i}\otimes x_k.$$ Then we get $$f^{\bar
p-1}_{\a_i}\otimes yx_m+\sum _{l<\bar p-1}f^l_{\a_i}\otimes w'_l=0,
w'_l\in W_i.$$ This gives us $yx_m=0$, so that $x_m\in
W_i^{\mathfrak{N}_i}$. Hence $x\in u(\mathfrak{P}_i,\chi)\otimes
_{u(\mathfrak{B}_i,\chi)}W_i^{\mathfrak{N}_i}$.\par

Let $F_i=(F_i)_{\bar{0}}\oplus (F_i)_{\bar{1}}$. Since
$(F_i)_{\bar{0}}^{[p]}=0$ and $\chi(F_i)=0$, $u(F_i)\subseteq
u(\g,\chi)$ is  supersymmetric. \par  Let $\mathbf F v_L$ be the
unique 1-dimensional trivial submodule  in the left regular module
$u(F_{i-1})$. We may assume that $v_L=f^{\bar{p}-1}_{\a_{j_1}}\dots
f^{\bar{p}-1}_{\a_{j_{i-1}}}$, where $\a_{j_1}$,\dots
,$\a_{j_{i-1}}$ is the set $\a_1,\dots, \a_{i-1}$ put in the order
of ascending heights.  By Lemma 5.9, $v_R=v_L$.\par By definition,
$W_i$ is a free $u(F_{i-1})$-module. In particular, as a
$u(F_{i-1})$-submodule,
$$W_i\cong u(F_{i-1})\otimes_{\mathbf F} W_1\quad(\text{canonically imbedded in}
\quad u(\mathfrak{P},\chi)\otimes_{u(\mathfrak{B},\chi)}W_1).$$ Then
the fact $v_L=v_R$ shows that
$$W_i^{F_{i-1}}=v_L\otimes W_1.$$Since $F_{i-1}\subseteq \mathfrak{N}_i$,
$W_i^{\mathfrak{N}_i}\subseteq W^{F_{i-1}}_i$.  It follows that $
W^{\mathfrak{N}_i}_i\subseteq (v_L \otimes
W_1)^{\mathfrak{N}_i}.$\par For any $\a\in \Phi^+_{\mathfrak l}$ and
$w_1\in W_1$, if $e_{\a}v_L\otimes w_1=0$, then we get $v_L\otimes
e_{\a}w_1=0$ by Lemma 5.11.   Hence  $e_{\a} w_1=0$. Since
$\mathfrak{N}=\sum _{\a\in \Phi'}\g_{\a}$ is the nilradical of
$\mathfrak{P}$ and $\chi(\mathfrak{N})=0$, Lemma 2.1(2) shows that
$\mathfrak{N}W_1=0$. Note that $\mathfrak{N}_1\subseteq
\oplus_{\a\in\Phi^+_{\mathfrak l}}\g_{\a}+\mathfrak{N}$. So we get
$\mathfrak{N}_1 w_1=0$. This gives us
$$
W_i^{\mathfrak{N}_i}\subseteq v_L\otimes W_1^{\mathfrak{N}_1}.
$$
Since $\mathfrak{N}_i$ is an ideal of $\mathfrak{B}_i$ and
$e_{\a_i}\in \mathfrak{B}_i$, $e_{\a_i}W^{\mathfrak{N}_i}_i\subseteq
W^{\mathfrak{N}_i}_i$. Thus we get
$$
e_{\a_i}W^{\mathfrak{N}_i}_i\subseteq W^{\mathfrak{N}_i}_i\cap
e_{\a_i}v_L\otimes W^{\mathfrak{N}_1}_1.
$$
It follows that
$$
e_{\a_i}W^{\mathfrak{N}_i}_i\subseteq v_L\otimes
W^{\mathfrak{N}_1}_1\cap e_{\a_i}v_L\otimes
W^{\mathfrak{N}_1}_1\\
$$$$=v_L\otimes W^{\mathfrak{N}_1}_1\cap
(v_L\otimes e_{\a_i}W^{\mathfrak{N}_1}_1+[e_{\a_i},v_L]\otimes
W^{\mathfrak N_1}_1)$$$$=v_L\otimes W^{\mathfrak{N}_1}_1\cap
[e_{\a_i},v_L]\otimes W^{\mathfrak N_1}_1.
$$
 By Lemma 5.12(2), $\bar P_{i-1}=:\mathfrak{B}+F_{i-1}$ is a parabolic subalgebra
 of $\g$. Clearly $e_{\a_i}$ is in the nilradical of $\bar P_{i-1}$:
 $\bar N_{i-1}=\oplus_{\a\in\Phi_{\mathfrak{l}}^+}\g_{\a} +\langle e_{\a_i},\dots,e_{\a_s}\rangle$. Therefore,
 we have $[e_{\a_i}, f_{\a_j}]\in \bar N_{i-1}\subseteq \mathfrak{B}$ for each $j<i$.
 It follows that $$[e_{\a_i},v_L]\otimes
 W_1\subseteq \sum_{l_k<\bar p-1\quad \text{for some}\quad k}e^{l_1}_{\a_{j_1}}\cdots
 e^{l_{i-1}}_{\a_{j_{i-1}}}\otimes  W_1.$$ Then we get
 $e_{\a_i}W^{\mathfrak{N}_i}_i=0$.\par

 We  shall now prove that
$$(2)\quad (u(\mathfrak{P}_i,\chi)\otimes
_{u(\mathfrak{B}_i,\chi)}W_i^{\mathfrak{N}_i})^{e_{\a_i}}=1\otimes
W^{\mathfrak{N}_i}_i.$$ Let $$x=\sum_{0\leq k\leq \bar
p-1}f^k_{\a_i}\otimes w_k\in (u(\mathfrak{P}_i,\chi)\otimes
_{u(\mathfrak{B}_i,\chi)}W_i^{\mathfrak{N}_i})^{e_{\a_i}}.$$ Then we
get $$0=e_{\a_i}\sum_{0\leq k\le \bar p-1}f^k_{\a_i}\otimes w_k$$
$$
=\sum_{1\leq k\leq \bar p-1}[e_{\a_i},f^k_{\a_i}]\otimes w_k$$$$=
-\sum_{1\leq k\leq \bar p-1}2kf^{k-1}_{\a_i}\otimes
(k-1-h_{\a_i})w_k.$$ This gives $$(k-1-h_{\a_i})w_k=0\quad
\text{for}\quad 1\leq k\leq \bar p-1.$$ i.e.,
$h_{\a_i}w_k=(k-1)w_k$. It follows that
$$\chi(h_{\a_i})^pw_k=(h^p_{\a_i}-h_{\a_i})w_k=0, 0\leq k\leq \bar p-1.$$
So we obtain $w_k=0$,
 $1\leq k\leq \bar p-1$ and  hence $x=1\otimes w_0\in
W_i^{\mathfrak{N}_i}.$\par From (1) and (2), we get
$$
0\neq ((W'\cap W_{i+1})^{\mathfrak{N}_i})^{e_{\a_i}}\subseteq
(u(\mathfrak{P}_i,\chi)\otimes
_{u(\mathfrak{B}_i)}W^{\mathfrak{N}_i}_i)^{e_{\a_i}}=1\otimes
W^{\mathfrak{N}_i}_i.
$$
This implies that   $W'\cap W_i\neq 0$, a contradiction.\par Then we
must have $W'\cap (1\otimes W_1)\neq 0$, and hence $W'=W$. Therefore
$W$ is a simple $u(\g,\chi)$-module, so that $M\cong W$.\par Since
$\mathfrak{N}$ is the nilradical of $\mathfrak{P}$ with
$\chi(\mathfrak{N})=0$, Lemma 2.1(2) shows  that each simple
$u(\mathfrak{P},\chi)$-module is a simple
$u(\mathfrak{l}',\chi)$-module and vice versa. So $M'$ is a simple
$u(\mathfrak{l}',\chi)$-module. Let
$\mathfrak{l}'=\mathfrak{l}\oplus H'$, where $H'\subseteq H$ and
$[H', \mathfrak{l}']=0$. Using the fact that $H'\subseteq \g_{\0}$,
 Schur's Lemma (\cite[p.18]{k1}) shows  that $H'$ acts as scalar
multiplications on $M'$. Thus, $M'$ is a simple
$u(\mathfrak{l},\chi)$-module. Note that $\chi|_{\mathfrak{l}}$ is
nilpotent.
\end{proof}

\subsection{Standard Levi form}

The study of the simple $u(\g,\chi)$-modules is now reduced to the
case where $\chi$ is nilpotent. Let $\mathfrak{B}$ be a Borel
subalgebra of $\g$.
\begin{definition}We say that $\chi$ has standard Levi form if
$\chi(\mathfrak{B})=0$ and there is a subset $I\subseteq \Delta_0$
 such that

$$
\chi(f_{\a})\begin{cases}\neq 0 &\text{if $\a\in I$,}\\
=0 & \text{if $\a\in \Phi\setminus I$.}\end{cases} $$

\end{definition}
\begin{proposition}If $\chi$ is in the standard Levi form, then each
$Z^{\chi}(\l)$ has a unique maximal submodule.
\end{proposition}
The proof is similar to that of \cite[Prop. 10.2]{j}.\par  Each
nilpotent p-character is conjugate under $G$ to the standard Levi
form. Let $\chi$ be in the standard Levi form and let $I$ be the set
of simple roots as above. For each $\a\in I$, let $\l(h_\a)=a$. Then
since $\l(h_{\a})^p-\l(h_{\a})=0$, we get $a\in \mathbf{F}_p$. Write
$0\leq a\leq p-1$.  Then it is easy to see that $f^{a+1}_{\a}v$ is
another maximal vector in $Z^{\chi}(\l)$ with weight $\l-(a+1)\a$.
So we get
$$ Z^{\chi}(\l-(a+1)\a)\cong Z^{\chi}(\l).$$ Recall $\rho$ defined
earlier. We define the dot action of $w\in W$ on $H^*$ by
$$
w\cdot \l=w(\l+\rho)-\rho.
$$
If $\a\in \Delta_0$ is simple, then we have $$s_{\a}\cdot
\l=\l-(\l(h_{\a})+1)\a.$$ Denote $W_I$ the subgroup of $W$ generated
by the reflections of $s_{\a}$ with $\a\in I$. Let $L^{\chi}(\l)$ be
the unique simple quotient of $Z^{\chi}(\l)$.
\begin{corollary}Let $w\in W_I$, then \par(1) $$Z^{\chi}(w\cdot\l)\cong Z^{\chi}(\l).$$
(2) $$L^{\chi}(w\cdot \l)\cong L^{\chi}(\l).$$
\end{corollary}

\def\refname{

\end{document}
\centerline{\bf REFERENCES}}
\begin{thebibliography}{10}


\bibitem{c} C. W. Curtis and I. Reiner. Representation theory of
finite groups and associative algebras. John Wiley $\&$ Sons,
(1962).
\bibitem{d} J. Dixmier. Enveloping algebras,  \textit{Graduate studies
in Math.} \textbf{11} \textit{Amer. Math. Soc.}, (1996).

\bibitem{f} E. M. Friedlander. B. J. Parshall. Modular
Representation theory of Lie algebras. \textit{Amer. J. Math.}
\textbf{110} (1988): 1055-1093.
\bibitem{ja} N. Jacobson. Lie algebras, Dover Publications, Inc. New
York,(1962).
\bibitem{j} J. C. Jantzen. Representations of Lie algebras in prime
characteristic. \textit{Proc. Montreal (NATO ASI series C)}
\textbf{514} (1997).
\bibitem{k1} V. G. Kac. Lie superalgebras. \textit{Adv. Math.}
 \textbf{26} (1977): 8--96.

\bibitem{k2} V. G. Kac. Classification of infinite-dimensional
simple linearly compact Lie superalgebras. \textit{Adv. Math.}
\textbf{139} (1998): 1--55.
\bibitem{k3} V. Kac and B. Weisfeiler. Coadjoint action of a semisimple algebraic group
and the center of the enveloping algebra in characteristic p.
\textit{Indag. Math.} \textbf{38} (1976): 135-151.
\bibitem{kl} Yu. Kochetkov and D. Leites. Simple Lie algebras in
characteristic 2 recovered from superalgebras and on the notion of a
simple finite group. \textit{Contemp. Math.}, \textbf{131}, Part 2.
\textit{Amer. Math. Soc. }RI, (1992).

\bibitem{p}   V. M. Petrogradski.  Identities in the enveloping algebras for
modular Lie superalgebras. \textit{J. Algebra} \textbf{145} (1992):
1--21.
\bibitem{r}  A. N. Rudakov. On representations of classical
semisimple Lie algebras of characteristic p. \textit{Izv. Akad. Nauk
SSSR} Ser. Mat.Tom \textbf{34} No.4 (1970).
 \bibitem{sc} M. Scheunert. \textit{Theory of Lie superalgebras}, in: Lecture
Notes in Math. Vol. 716. Springer-verlag (1979).
\bibitem{s} H. Strade. \textit{Simple Lie algebras over fields of
positive characteristic, I. Structure Theory} Walter de Gruyter.
(2004).
\bibitem{sf} H. Strade and R. Farnsteiner. \textit{Modular Lie Algebras and Their
Representations}, in: Monogr. Texbooks Pure Appl. Math. Vol. 116.
Dekker, Inc. (1988).

\bibitem{h} B. Weisfeiler and V. Kac. Ireducible representations of Lie p-algebras
 \textit{Funks. Anal} \textbf{5} (1971): 28-36.
\bibitem{zl}  Chaowen Zhang.
\textit{Simple modules for the restricted Lie superalgebras
$sl(n,1)$} \textit{J.  Pure and applied algebra}, \textbf{213}, No.
5, (2009) 756-765.

\end{thebibliography}
